\documentclass[12pt]{article}%
\usepackage{amsmath}
\usepackage{amsfonts}
\usepackage{amssymb}
\usepackage{graphicx}%
\providecommand{\U}[1]{\protect\rule{.1in}{.1in}}
\newtheorem{theorem}{Theorem}

\newtheorem{corollary}[theorem]{Corollary}

\newtheorem{proposition}[theorem]{Proposition}
\newtheorem{remark}[theorem]{Remark}

\newenvironment{proof}[1][Proof]{\noindent\textbf{#1.} }{\ \rule{0.5em}{0.5em}}
\begin{document}

\title{A note on compact and $\sigma$-compact subsets of probability measures on
metric spaces with an application to the distribution free newsvendor problem}
\author{\'{O}scar Vega-Amaya\thanks{This author (OVA) declares that he has no conflict
of interest.} \ \thanks{Corresponding author. Email: oscar.vega@unison.mx}
\ and Fernando Luque-V\'{a}squez\thanks{This author (FLV) declares that he has
no conflict of interest.} \ \thanks{Email: fernando.luque@unison.mx}\\Departamento de Matem\'{a}ticas,\\Universidad de Sonora,\\Hermosillo, Sonora, M\'{e}xico}
\date{January 18, 2023}
\maketitle

\begin{abstract}
This note identifies compact and $\sigma$-compact subsets of probability
measures on a class of metric spaces with respect to the weak convergence
topology. Moreover, it is shown by an example, that the space of probability
measures on a $\sigma$-compact metric spaces not need to be $\sigma$-compact
space, even though the converse statement holds true for metric spaces. The
results are applied to an extended form of the distribution free newsvendor problem.

\textbf{Key words:} probability measures spaces, weak convergence, weak
$\sigma$-compact subsets, minimax problems, newsvendor problem.

\end{abstract}

\section{Introduction}

This note shows the compactness and $\sigma$-compactnes of some important
subsets of probability measures on (Heine-Borel) metric spaces with respect to
the weak convergence of measures. These kind of results, in addition to be
interesting by themselves, are useful in minimax or distributional robustness
problems, that is, in optimization problems subject to random factors where
the involved probability distribution is partiallly known or misspecified
(see, for instance, \cite{GT, Kara, ova}). In fact, these results are used in
section 4 to extend the validity of the Scarf's rule \cite{Gallego, Scarf} for
the distribution free newsvendor problem with known mean and variance to the
case where these quantities are just known that belong to some closed
intervals. It is also shown with a very short and simple proof that the
compactness of the metric space implies the compactness of the space of
probability measures.

On the other hand, it is well known that a metric space is compact if and only
if the set of probability measures defined on its Borel $\sigma$-algebra is a
compact space when endowed with the topology of the weak convergence of
measures (see, for instance, \cite[Thm. 15.11, p. 513]{Ali}). Thus, one can
wonder whether a similar result holds for $\sigma$-compact spaces, that is, if
the $\sigma$-compactness of a metric space implies the $\sigma$-compactness of
the space of probability measures and viceversa. This problem was already
raised in reference \cite{Luque}, remaining open up to now to the best
knowledge of the authors. The present note shows that the set of probability
measures on the set of real numbers--endowed with the standard metric--is not
$\sigma$-compact (Proposition \ref{T3} below), so the first statement is
false. More specifically, it is shown that any closed ball of probability
measures on the real number set is not compact; thus, any compact subset of
probability measures is nowhere dense, which means that it has not interior
points. Hence, by the Baire's category theorem (see, for instance, \cite[Thm.
4.7-2, p. 247]{Krey}), the space of probability measures can not be $\sigma
$-compact because it is a complete metric space. On the other hand,
Proposition \ref{T4} shows that the reciprocal stament holds for any metric
space, that is, the $\sigma$-compactness of the probability measures implies
the $\sigma$-compactness of the metric space.

The remainder part of the present note is organized as follows. To ease the
reading, section 2 collects a number of concepts and results related to the
weak convergence of probability measures, whereas section 3 introduces the
subsets of probability measures of interest and proves their compactness
properties. Section 4 starts discussing briefly the classical newsvendor
problem and then continues with the distribution free variant introduced by
Scarf \cite{Scarf} and the extension previously mentioned.

\section{Preliminary concepts and results}

For a metric space $(X,d),$ denote by $B_{d}(x,r)$ and $B_{d}[x,r]$ the open
and closed balls, respectively, with center at the point $x\in X$ and radius
$r>0.$ The class of bounded continuous functions on $X$ is denoted by
$\mathcal{C}_{b}(X)$ and the Borel $\sigma$-algebra by $\mathcal{B}(X)$.
Recall that the Borel $\sigma$-algebra is generated by the open subsets of
$X.$ Moreover, denote by $\mathbb{P}(X)$ the space of probability measures on
$X$ endowed with the topology of the weak convergence $\mathcal{W}$, that is,
with the coarsest topology that makes continuous the mapping%
\[
\mu\rightarrow\int_{X}v(y)\mu(dy)
\]
for all $v\in\mathcal{C}_{b}(X)$. Thus, it is said that a sequence $\{\mu
_{n}\}\subset\mathbb{P}(X)$ converges weakly to $\mu\in\mathbb{P}(X)$ (written
as $\mu_{n}\Rightarrow\mu$ for short) if and only if%
\[
\int_{X}v(y)\mu_{n}(dy)\rightarrow\int_{X}v(y)\mu(dy)\ \ \ \forall
v\in\mathcal{C}_{b}(X).
\]

The weak convergence topology $\mathcal{W}$ is metrizable and separable if and
only if the space $(X,d)$ is a separable metric space (see \cite[Thm. 5.12, p.
513]{Ali} or \cite[Prop. 7.20, p. 127]{Bert}). Thus, if the space $(X,d)$ is a
separable metric space, the so-called Prokhorov metric metrizes the topology
$\mathcal{W}$ \cite[Thm 11.3.1, p. 394]{Dudley}. The Prokhorov metric--also
known as L\'{e}vy-Prokhorov metric--is defined as follows. For a subset
$A\subset X$ and $\varepsilon>0$ put%
\[
A^{\varepsilon}:=\{x\in X:d(x,y)<\varepsilon\text{ for some }y\in A\}.
\]
The Prokhorov metric is defined as%
\[
d_{P}(\mu,\lambda):=\inf\{\varepsilon>0:\mu(A)\leq\lambda(A^{\varepsilon
})+\varepsilon\ \ \forall A\in\mathcal{B}(X)\}
\]
for $\mu,\lambda\in\mathbb{P}(X)$. Moreover, the metric space $(X,d)$ is a
Polish space (separable complete metric space) if and only if $\mathbb{P}(X)$
is a Polish space (\cite[Thm. 15.15, p. 515]{Ali}).

For the real numbers set $\mathbb{R}$ endowed with the standard metric, the
weak convergence topology $\mathcal{W}$ is also metrized by the L\'{e}vy
metric given next. Denote by $F_{\mu}$ the probability distribution function
defined by the probability measure $\mu\in\mathbb{P}(\mathbb{R})$. The
L\'{e}vy metric is given as%
\[
d_{L}(\mu,\lambda):=\inf\{\varepsilon>0:F_{\mu}(x-\varepsilon)-\varepsilon\leq
F_{\lambda}(x)\leq F_{\mu}(x+\varepsilon)+\varepsilon\ \ \forall
x\in\mathbb{R}\}.
\]

Thus, $\mu_{n}\overset{w}{\rightarrow}\mu$ if and only if $d_{L}(\mu_{n}%
,\mu)\rightarrow0$ \cite[p. 423]{Gibbs}. Moreover, such a convergence is
equivalent to the weak convergence of the probability distribution functions,
that is, to%
\[
F_{\mu_{n}}(x)\rightarrow F_{\mu}(x)\ \ \forall x\in C_{F_{\mu}},
\]
where $C_{F}$ stands for the subset of continuity points of the probability
distribution function $F$ (see \cite[Helly-Bray Theorem 11.1.2, p.
387]{Dudley} or \cite[p. 18]{Bill}).

\section{Compact and $\sigma$-compact subsets of probability measures in
metric spaces}

Consider the subset of probability measures%
\[
\mathbb{P}_{b}(X):=\{\mu\in\mathbb{P}(X):\int_{X}d(x,x_{0})\mu(dx)\leq b\},
\]
where $x_{0}\in X$ is a fixed but arbitrary point and $b>0,$ and also the
subset%
\[
\mathbb{P}_{0}(X):=\{\mu\in\mathbb{P}(X):\int_{X}d(x,x_{0})\mu(dx)<\infty\}.
\]
Next, for constants $0\leq a<b,c>0$ and $r>0,$ let%
\[
\mathbb{P}_{a,b}(X):=\{\mu\in\mathbb{P}(X):a\leq\int_{X}d(x,x_{0})\mu(dx)\leq
b\},
\]%
\[
\mathbb{P}_{a,b}^{r,c}(X):=\{\mu\in\mathbb{P}_{a,b}(X):\int_{X}d^{1+r}%
(x,x_{0})\mu(dx)\leq c\}
\]
and%
\[
\mathbb{P}_{a,b}^{r}(X):=\{\mu\in\mathbb{P}_{a,b}(X):\int_{X}d^{1+r}%
(x,x_{0})\mu(dx)<\infty\}.
\]

Theorem \ref{T1} below shows the compactness of the subset $\mathbb{P}_{b}%
(X)$; thus, it follows that $\mathbb{P}_{0}(X)$ is $\sigma$-compact. This
result is borrowed from \cite[Thm. 4]{Luque} and it uses the following
concept: a metric space $(X,d)$ is said to be \textit{Heine-Borel} or
\textit{proper} metric space if the closed and bounded subsets are compact
(\cite{Will}, \cite[ Ch. 9, Problem 31]{Royden}). For instance, the set of
real numbers $\mathbb{R}$ endowed with the usual metric is a Heine-Borel
metric space. As a direct consequence of Theorem \ref{T1}, Corollary \ref{T2}
below shows that the compactness of the metric space $X$ implies the
compactness of the space $\mathbb{P}(X)$. On the other hand, Proposition
\ref{T5} below shows that $\mathbb{P}_{a,b}(X)$ for $a>0$ not need be a closed
subset of probability measures, whereas Theorem \ref{T6} proves that
$\mathbb{P}_{a,b}^{r,c}(X)$ is compact (hence, $\mathbb{P}_{a,b}^{r}(X)$ is a
$\sigma$-compact subset).

\begin{theorem}
\label{T1}If $(X,d)$ is a Heine-Borel metric space, then $\mathbb{P}_{b}(X)$
is a compact subset for each $b>0.$ Hence, $\mathbb{P}_{0}(X)$ is a $\sigma
$-compact subset.
\end{theorem}

\begin{proof}
Let $b>0$ and $\varepsilon>0$ be arbitrary; next consider the compact subset
$K:=B_{d}[x_{0},2b/\varepsilon]$ and the measurable function $Z(\cdot
):=d(\cdot,x_{0})$ defined on the measure space $(X,\mathcal{B}(X),\mu),$
where $\mu$ is an arbitrary probability measure in $\mathbb{P}_{b}(X).$ Thus,
the Markov inequality implies that%
\[
\mu(K^{c})\leq\mu(Z>\tfrac{2b}{\varepsilon})\leq\dfrac{\varepsilon}%
{2}<\varepsilon.
\]
Then, the set $\mathbb{P}_{b}(X)$ is a class of (uniformly) tight measures.
Hence, by Prokhorov theorem \cite[Thm. 6.1, p. 37]{Bill}, $\mathbb{P}_{b}(X)$
is relatively compact, that is, any sequence in $\mathbb{P}_{b}(X)$ has a
weakly convergent subsequence. Thus, to show that $\mathbb{P}_{b}(X)$ is a
compact subset it is enough to prove that this subset is closed.

Suppose that the sequence $\{\mu_{n}\}\subset\mathbb{P}_{b}(X)$ converges to
$\mu\in\mathbb{P}(X).$ Consider the lower semicontinuous functions
$Z_{k}:=Z\mathbb{I}_{B(x_{0},2k/\varepsilon)},k\in\mathbb{N}$. Here,
$\mathbb{I}_{A}$ stands for the indicator function of the subset $A\subset X.$
Then%
\[
\int_{X}Z_{k}(x)\mu(dx)\leq\liminf_{n\rightarrow\infty}\int_{X}Z_{k}(x)\mu
_{n}(dx)\leq\liminf_{n\rightarrow\infty}\int_{X}Z(x)\mu_{n}(dx)\leq b.
\]
Since $Z_{k}\uparrow Z$ pointwise, the latter inequalities yields%
\[
\int_{X}Z(x)\mu(dx)\leq b,
\]
which shows that $\mu\in\mathbb{P}_{b}(X).$ Therefore, $\mathbb{P}_{b}(X)$ is
a compact subset.
\end{proof}

\begin{corollary}
\label{T2}A metric space $(X,d)$ is compact if and only if $\mathbb{P}(X)$ is
a compact space.
\end{corollary}

\begin{proof}
Suppose that $(X,d)$ is a compact space. Then, it is obviously a Heine-Borel
metric space. Next observe that the constant $b^{\ast}:=\sup_{x\in X}%
d(x,x_{0})$ is finite since the mapping $d(\cdot,x_{0})$ is continuous and $X$
is compact$.$ Hence, $\mathbb{P}(X)=\mathbb{P}_{b^{\ast}}(X),$ which yields
the compactness of $\mathbb{P}(X)$.

The proof of the second part is the same given in \cite[Thm. 15.11, p.
513]{Ali}. It is included here just for the sake of completeness. Thus,
suppose that $\mathbb{P}(X)$ is a compact space and denote by $\delta(x)$ the
Dirac measure at $x\in X$, that is, the measure concentrated at the point
$x\in X.$ Now, from \cite[Thm. 15.8, p. 512]{Ali}, the mapping $x\rightarrow
\delta(x)$ embeds $X$ into $\mathbb{P}(X);$ hence, $X$ can be topologically
identified with the subset%
\[
\delta(X):=\{\delta(x):x\in X\}\subset\mathbb{P}(X).
\]
Then, observe that $\delta(X)$ is a\ separable and closed subset; hence, it is
compact. Therefore, $X$ itself is a compact space.
\end{proof}

Proposition \ref{T3} below shows that $\mathbb{P}(X)$ need not be a $\sigma
$-compact space when $X$ is a $\sigma$-compact space. Nonetheless, the
converse assertion holds true as shown in Proposition \ref{T4} below. The
former proposition uses the Baire's category theorem \cite[Thm. 4.7-2, p.
247]{Krey}, which states that no complete metric space is a denumerable union
of closed sets with empty interior.

\begin{proposition}
\label{T3}The probability measures space $\mathbb{P}(\mathbb{R})$ is not a
$\sigma$-compact space, where $\mathbb{R}$ is the real numbers set endowed
with the usual metric.
\end{proposition}

\begin{proof}
The key point to prove the result is that the closed balls of $\mathbb{P}%
(\mathbb{R})$ are not compact subsets; this implies that the compact subsets
of $\mathbb{P}(\mathbb{R})$ are nowhere dense, that is, they have empty
interior. Since $\mathbb{P}(\mathbb{R})$ is a complete metric space, from the
Baire's category theorem follows that it is not a $\sigma$-compact space
\cite[Thm. 4.7-2, p. 247]{Krey}.

Let $\mu\in\mathbb{P}(\mathbb{R})$ be a fixed probability measure and denote
by $F$ its probability distribution function. Next it is shown that the ball
$B_{d_{L}}[\mu,r]$ is not a compact subset for any $r>0$. Note that one can
assume without losses of generality that $r\in(0,1/2).$ Define the constants%
\begin{align*}
a_{r}  &  :=\sup\{x\in\mathbb{R}:F(x)<r\},\\
& \\
b_{r}  &  :=\inf\{x\in\mathbb{R}:F(x)\geq1-r\},
\end{align*}
and consider the probability distribution functions defined as follows: for
$n\leq\max\{|a_{r}|,|b_{r}|\},$ set%
\[
F_{n}(x):=F(x),\ x\in\mathbb{R}.
\]
For $n>\max\{|a_{r}|,|b_{r}|\},$ put%

\[
F_{n}(x):=\left\{
\begin{array}
[c]{ccc}%
0 & \text{for} & x<-n,\\
r & \text{for} & -n\leq x<a_{r},\\
F(x) & \text{for} & a_{r}\leq x<b_{r},\\
1-r & \text{for} & b_{r}\leq x<n,\\
1 & \text{for} & x\geq n.
\end{array}
\right.
\]
Denote by $\mu_{n}$ the probability measure corresponding to the probability
distribution function $F_{n}.$

Recall that the Kolmogorov metric on $\mathbb{P}(\mathbb{R})$ is given as%
\[
d_{K}(\mu,\lambda):=\sup_{x\in\mathbb{R}}|F_{\mu}(x)-F_{\lambda}(x)|
\]
and also that $d_{L}\leq d_{K}\ $(see \cite[p. 425]{Gibbs}). Moreover, it can
be seen by direct computations that%
\[
d_{K}(\mu_{n},\mu)\leq r,
\]
which implies that the sequence $\{\mu_{n}\}$ belongs to the closed ball
$B_{d_{L}}[\mu,r].$ Now notice that the sequence $\{\mu_{n}\}$ is not
(uniformly) tight; hence, by Prokhorov theorem \cite[Th, 6.2, p. 37]{Bill},
$B_{d_{L}}[\mu,r]$ is not a compact subset of $\mathbb{P}(\mathbb{R})$.
\end{proof}

\begin{proposition}
\label{T4}If $\mathbb{P}(X)$ is a $\sigma$-compact space, then $X$ is a
$\sigma$-compact space.
\end{proposition}

\begin{proof}
To prove this assertion first note that the $\sigma$-compactness property
implies that $\mathbb{P}(X)$ is separable, which in turn implies that $X$ is a
separable space by \cite[Thm. 15.12, p. 513]{Ali}. Now, let $\{\mathbb{K}%
_{n}\}$ be a sequence of compact subsets such that $\mathbb{P}(X)=\cup
_{n\geq1}\mathbb{K}_{n}$ and define%
\[
K_{n}:=\{x\in X:\delta(x)\in\mathbb{K}_{n}\},\ n\in\mathbb{N}.
\]
Notice that this latter subset is topologically the same that the subset
$\delta(X)\cap\mathbb{K}_{n}$, which is a compact subset of $\mathbb{P}(X)$
\cite[Thm. 15.8]{Ali}. Thus, the $\sigma$-compactness of $X$ follows after
noting that%
\[
X=\cup_{n\geq1}K_{n}.
\]

\end{proof}

\begin{proposition}
\label{T5}The subset $\mathbb{P}_{a,b}(\mathbb{R)},$ with $a>0,$ is not a
closed subset of $\mathbb{P}(\mathbb{R)}.$
\end{proposition}

\begin{proof}
Take $x_{0}=0,$ so
\[
\mathbb{P}_{a,b}(\mathbb{R})=\{\mu\in\mathbb{P(R)}:a\leq\int_{\mathbb{R}%
}\left\vert x\right\vert \mu(dx)\leq b\}.
\]

Let $\{\mu_{n}\}$ be the sequence in $\mathbb{P(R)}$ defined as%
\[
\mu_{n}(\frac{na}{2n-1}):=(1-\frac{1}{2n}),\ \ \ \ \ \ \mu_{n}(na):=\frac
{1}{2n},\ \ n\in\mathbb{N}.
\]
Observe that%
\[
\int_{\mathbb{R}}\left\vert x\right\vert \mu_{n}(dx)=(\frac{na}{2n-1}%
)(1-\frac{1}{2n})+(na)(\frac{1}{2n})=a;
\]
thus, $\mu_{n}\in\mathbb{P}_{a,b}$ all $n\in\mathbb{N}.$

Next, consider the sequence of probability distribution functions $\{F_{n}\}$
corresponding to the sequence of probability measures $\{\mu_{n}\}.$ Clearly,
$F_{n}(x)\rightarrow F(x)$ for $x\neq a/2$ where
\[
F(x):=\left\{
\begin{array}
[c]{c}%
0,\ \ \ \text{if \ \ }x<a/2,\\
1,\ \ \ \ \text{if \ \ }x\geq a/2.
\end{array}
\right.
\]
Thus, $\mu_{n}\Rightarrow\mu$ where $\mu$ is the probability measure
corresponding to probability distribution function $F.$ Notice that $\mu
\notin\mathbb{P}_{a,b}(\mathbb{R)}$; hence, $\mathbb{P}_{a,b}(\mathbb{R)}$ is
not a closed subset.
\end{proof}

\begin{theorem}
\label{T6}If $(X,d)$ is a Heine-Borel metric space, then $\mathbb{P}%
_{a,b}^{r,c}(X)$ is a compact subset for $0<a<b,r>0$ and $c>0$. Hence,
$\mathbb{P}_{a,b}^{r}(X)$ is a $\sigma$-compact subset.
\end{theorem}

\begin{proof}
Using similar arguments to the proof of Theorem \ref{T1}, one can prove that
$\mathbb{P}_{a,b}^{r,c}(X)$ is tight. Then, by Prokhorov theorem, it suffices
to prove that $\mathbb{P}_{a,b}^{r,c}(X)$ is a closed subset. Thus, let
$\{\mu_{n}\}\subset\mathbb{P}_{a,b}^{r,c}(X)$ be a sequence that converges to
a probability measure $\mu\in\mathbb{P}(X)$. Proceeding again as in the proof
of Theorem \ref{T1}, it follows that%
\[
\int_{X}d(x,x_{0})\mu(dx)\leq b\ \ \ \text{and}\ \ \ \int_{X}d(x,x_{0}%
)^{1+r}\mu(dx)\leq c.
\]
Hence, it only remains to prove that%
\[
\int_{X}d(x,x_{0})\mu(dx)\geq a.
\]

Next, notice that the following inequalities hold for $\lambda\in
\mathbb{P}_{a,b}^{r,c}(X)$:%
\begin{align*}
\int_{\{x\in X:d(x,x_{0})\geq k\}}d(x,x_{0})\lambda(dx)  &  =\int_{\left\{
x\in X:\left(  \frac{d(x,x_{0})}{k}\right)  ^{r}\geq1\right\}  }%
d(x,x_{0})\lambda(dx)\\
& \\
&  \leq\int_{\left\{  x\in X:\left(  \frac{d(x,x_{0})}{k}\right)  ^{r}%
\geq1\right\}  }\frac{(d(x,x_{0}))^{1+r}}{k^{r}}\lambda(dx)\\
& \\
&  \leq\frac{c}{k^{r}}.
\end{align*}
Thus, for each $\varepsilon>0$ there exists $k_{\varepsilon}>0$ such that%
\[
\int_{\{x\in X:d(x,x_{0})\geq k_{\varepsilon}\}}d(x,x_{0})\lambda
(dx)<\varepsilon\ \ \ \forall\lambda\in\mathbb{P}_{a,b}^{r,c}(X).
\]

On the other hand, note that the mapping $x\rightarrow d(x,x_{0}%
)\mathbb{I}_{B_{d}[x_{0},k_{\varepsilon}]}(x)$ is upper semicontinuous. Then,%
\begin{align*}
\int_{X}d(x,x_{0})\mu(dx)  &  \geq\int_{X}d(x,x_{0})\mathbb{I}_{B_{d}%
[x_{0},k_{\varepsilon}]}(x)\mu(dx)\\
& \\
&  \geq\limsup_{n\rightarrow\infty}\int_{X}d(x,x_{0})\mathbb{I}_{B_{d}%
[x_{0},k_{\varepsilon}]}(x)\mu_{n}(dx)\\
& \\
&  =\limsup_{n\rightarrow\infty}\left(  \int_{X}d(x,x_{0})\mu_{n}(dx)-\int%
_{X}d(x,x_{0})\mathbb{I}_{\{x\in X:d(x,x_{0})>k_{\varepsilon}\}}\mu
_{n}(dx)\right) \\
& \\
&  \geq\limsup_{n\rightarrow\infty}\int_{X}d(x,x_{0})\mu_{n}(dx)-\varepsilon\\
& \\
&  \geq a-\varepsilon.
\end{align*}
Since $\varepsilon$ is arbitrary, it follows that%
\[
\int_{X}d(x,x_{0})\mu(dx)\geq a,
\]
which completes the proof.
\end{proof}

\section{The distribution free newsvendor problem}

The newsvendor o newsboy problems constitute a family of optimization problems
widely studied in the field of operations research \cite{Choi, Gallego, Qin}.
The problem in its classical version is to determine the optimal stock
quantity $x^{\ast}\geq0$\ of a perishable product that a retailer has to ask
to a suplier periodically, say, daily, in order to face a random demand
$W\geq0$ with known probability distribution $\mu$. Thus, denoting by $c>0$
the unit purchase cost, by $p>0$ the selling price and by $q\geq0$ the unit
salvage price for the unsold products, the problem is to find the stock
quantity $x^{\ast}$ that maximize the expected reward%
\[
\pi(x,\mu):=pE_{\mu}\min(x,W)+qE_{\mu}\max(x-W,0)-cx,\ x\geq0,
\]
where $E_{\mu}$ stands for the expectation with respect to the probability
distribution measure $\mu$ of the random demand. Assuming that $p>c>q$, the
optimal stocking quantity is%
\[
x^{\ast}=\inf\{x\geq0:F(x)\geq\tfrac{p-c}{p-q}\},
\]
where $F$ is the probability distribution function defined by $\mu.$ In
particular, if $F$ is continuous and strictly increasing, then $x^{\ast
}=F^{-1}((p-c)/(p-q)).$

The above solution is very appealing but it assumes that the demand
distribution is completely known, which rarely ocurs in practice. Determining
the exact demand distribution is quite difficult or even impossible due to the
lack structural properties of the distribution and, in many cases, the
information available is just restricted to past observations that hopefully
allow good estimations of the mean and the variance. This give rise to the
problem known as \textquotedblleft distribution free newsvendor
problem\textquotedblright\ where the mean and variance are known but the
distribution itself is not. The standard approach to this problem is to
maximize the expected profits (or to minimize the expected cost) under the
worse possible distribution. More precisely, denoting by $\mathcal{F}%
(m,s^{2})$ the class of probability distribution measures on the set of
nonnegative real numbers $\mathbb{R}_{+}$ with finite mean $m$ and finite
variance $s^{2},$ the problem is to find the stock quantity $x^{\ast}\geq0$
such that%
\begin{equation}
\inf_{\mu\in\mathcal{F}(m,s^{2})}\pi(x^{\ast},\mu)=\sup_{x\geq0}\inf_{\mu
\in\mathcal{F}(m,s^{2})}\pi(x,\mu). \label{S1}%
\end{equation}

The above version of the newsvendor problems was introduced by Scarf
\cite{Scarf}, who gave a closed form and readily computable solution, namely,%
\begin{equation}
x^{\ast}=m+\frac{s^{2}}{2}\left(  \sqrt{\frac{p-c}{c-q}}-\sqrt{\frac{c-q}%
{p-c}}\right)  . \label{S2}%
\end{equation}
The above solution has became known as the Scarf's rule, and next it is
briefly discussed. A very detailed discusion can be found in reference
\cite{Gallego} together a number of extensions of the newsvendor problem with
unknown distribution. After that, it is shown the validity of the Scarf's rule
for the case where it is only known that the demand distribution $\mu$ has a
mean belonging to a closed interval and the variance is bounded above, that
is, $\mu$ belongs to a subset $\mathbb{P}_{a,b}^{1,c}(\mathbb{R}_{+}).$

Let $m(\mu)$ and $s^{2}(\mu)$ be the mean and the variance of the probability
measure $\mu$, respectively, whenever these quantities are finite. After some
elementary computation, one can verify that%
\[
\pi(x,\mu)=(p-q)m(\mu)-P(x,\mu)
\]
for all $x\geq0,\mu\in\mathbb{P}(\mathbb{R}_{+})$, where%
\[
P(x,\mu):=(c-q)x+(p-q)E_{\mu}\max(W-x,0).
\]
Thus, $x^{\ast}\geq0$ satisfies (\ref{S1}) if and only if%
\[
\sup_{\mu\in\mathcal{F}(m,s^{2})}P(x^{\ast},\mu)=\inf_{x\geq0}\sup_{\mu
\in\mathcal{F}(m,s^{2})}P(x,\mu).
\]
The next remark gives the key points that lead to the optimality of the
Scarf's rule (\ref{S2}).

\begin{remark}
\label{T7}(c.f. \cite{Scarf}, \cite{Gallego}). Let $m\geq0$ and $s\geq0.$ Then:

(a) $E_{\mu}[\max(W-x,0)]\leq\dfrac{1}{2}[\sqrt{s^{2}+(x-m)^{2}}-(x-m)]$ for
all $\mu\in\mathcal{F}(m,s^{2})$ and $x\geq0.$

(b) For each $x\geq0$ there exists a unique distribution $\overline{\mu}%
_{x}\in\mathcal{F}(m,s^{2})$ such that%
\[
E_{\overline{\mu}_{x}}[\max(W-x,0)]=\dfrac{1}{2}[\sqrt{s^{2}+(x-m)^{2}%
}-(x-m)];
\]
in fact, the distribution $\overline{\mu}_{x}$\ is concentrated at two points.

(c)\ Then%
\begin{align*}
\inf_{x\geq0}\sup_{\mu\in\mathcal{F}(m,s^{2})}P(x,\mu)  &  =\inf_{x\geq
0}L(x,m,s^{2})\\
& \\
&  =L(x^{\ast},m,s^{2}),
\end{align*}

where
\[
L(x,m,s^{2}):=(c-q)x+\frac{p-q}{2}\left[  \sqrt{s^{2}+(x-m)^{2}}-(x-m)\right]
.
\]

\end{remark}

The next theorem proves the optimality of the Scarf's rule when the
\textquotedblleft true\textquotedblright\ and unknown distribution $\mu$ has a
mean belonging to the interval $[a,b]$ and the variance $s^{2}$ is in the
interval $[0,d^{2}],$ where $0<a<b$ and $d>0.$ In other words, $\mu$ belongs
to the compact subset of measures $\mathbb{P}_{a,b}^{1,d^{2}+b^{2}}%
(\mathbb{R}_{+}).$

\begin{theorem}
\label{T8}Suppose that the probability distribution of the demand belongs to
$\mathbb{P}_{a,b}^{1,d^{2}+b^{2}}(\mathbb{R}_{+}).$Then the stock quantity%
\[
x^{\ast}=b+\frac{d^{2}}{2}\left(  \sqrt{\frac{p-c}{c-q}}-\sqrt{\frac{c-q}%
{p-c}}\right)
\]
satisfies the equalities%
\begin{align*}
\inf_{\mu\in\mathbb{P}_{a,b}^{1,d^{2}+b^{2}}(\mathbb{R}_{+})}\pi(x^{\ast}%
,\mu)  &  =\sup_{x\geq0}\inf_{\mu\in\mathbb{P}_{a,b}^{1,d^{2}+b^{2}%
}(\mathbb{R}_{+})}\pi(x,\mu)\\
& \\
&  =(c-q)x^{\ast}+\frac{p-q}{2}\left[  \sqrt{d^{2}+(x^{\ast}-b)^{2}}-(x^{\ast
}-b)\right]
\end{align*}

\end{theorem}

\begin{proof}
First note that for each $x\geq0$ the mapping%
\[
w\rightarrow p\min(x,w)+q\max(x-w,0)-cx
\]
is continuous on $\mathbb{R}_{+}$ and bounded by $px.$ Thus, $\pi(x,\cdot)$ is
continuous on $\mathbb{P(R}_{+})$\ for each $x\geq0$, that is, $\pi(x,\mu
_{n})\rightarrow\pi(x,\mu)$ whenever $\mu_{n}\Rightarrow\mu.$ Since
$\mathbb{P}_{a,b}^{1,d^{2}+b^{2}}(\mathbb{R}_{+})$ is a compact subset, for
each $x$ there exists a probability measure $\mu_{x}\in\mathbb{P}%
_{a,b}^{1,d^{2}+b^{2}}(\mathbb{R}_{+})$ such that%
\begin{align*}
\sup_{\mu\in\mathbb{P}_{a,b}^{1,d^{2}+b^{2}}(\mathbb{R}_{+})}\pi(x,\mu)  &
=\pi(x,\mu_{x})\\
& \\
&  =(p-q)m(\mu_{x})-P(x,\mu_{x})\\
& \\
&  \leq(p-q)m(\mu_{x})-\inf_{\mu\mathcal{\in F}(m(\mu_{x}),s^{2}(\mu_{x}%
))}P(x,\mu)
\end{align*}
for all $x\geq0.$ Now, from Remark \ref{T7}(b), there exists a unique
probability measure $\overline{\mu}_{x}\in\mathcal{F}(m(\mu_{x}),s^{2}(\mu
_{x})),x\geq0,$ such that%
\[
\inf_{\mu\mathcal{\in F}(m(\mu_{x}),s^{2}(\mu_{x}))}P(x,\mu)=P(x,\overline
{\mu}_{x}).
\]
Then,%
\begin{align*}
\pi(x,\mu_{x})  &  =\sup_{\mu\in\mathbb{P}_{a,b}^{1,d^{2}+b^{2}}%
(\mathbb{R}_{+})}\pi(x,\mu)\\
& \\
&  \leq(p-q)m(\mu_{x})-P(x,\overline{\mu}_{x})\\
& \\
&  =\pi(x,\overline{\mu}_{x})\\
& \\
&  \leq\sup_{\mu\in\mathbb{P}_{a,b}^{1,d^{2}+b^{2}}(\mathbb{R}_{+})}\pi
(x,\mu),
\end{align*}
which implies that $P(x,\mu_{x})=P(x,\overline{\mu}_{x}).$ Therefore, by the
uniqueness of $\overline{\mu}_{x},$ $\mu_{x}=\overline{\mu}_{x}$ for each
$x\geq0.$

On the other hand, it can be verified with direct computations that%
\[
\sup_{(m,s^{2})\in I}L(x,m,s^{2})=L(x,b,d^{2})\ \ \ \forall x\geq0,
\]
where $I:=[a,b]\times\lbrack0,d^{2}].$ Then, by Remark \ref{T7}(c),
\begin{align*}
\sup_{\mu\in\mathbb{P}_{a,b}^{1,d^{2}+b^{2}}(\mathbb{R}_{+})}P(x,\mu)  &
=L(x,m(\mu_{x}),s^{2}(\mu(x))\\
& \\
&  \leq L(x,b,d^{2})\\
& \\
&  =P(x,\mu^{\prime})
\end{align*}
for all $x\geq0,$ where $\mu^{\prime}$ is some probability distribution such
that $m(\mu^{\prime})=b$ and $s^{2}(\mu^{\prime})=d^{2}.$ Hence,%
\begin{align*}
\inf_{x\geq0}\sup_{\mu\in\mathbb{P}_{a,b}^{1,d^{2}+b^{2}}(\mathbb{R}_{+}%
)}P(x,\mu)  &  =\inf_{x\geq0}L(x,b,d^{2})\\
& \\
&  =L(x^{\ast},b,d^{2}),
\end{align*}
with
\[
x^{\ast}=b+\frac{d^{2}}{2}\left(  \sqrt{\frac{p-c}{c-q}}-\sqrt{\frac{c-q}%
{p-c}}\right)  .
\]

\end{proof}

\bigskip

\bigskip

\noindent\textbf{Compliance with Ethical Standards.} This work was partially
supported by Consejo Nacional de Ciencia y Tecnolog\'{\i}a (CONACYT-Mexico)
under grant Ciencia Frontera 2019-87787.\bigskip

\noindent\textbf{Ethical approval: }This article does not contain any studies
with human participants or animals performed by any of the authors.

\end{document}